\documentclass[12pt]{article}

\usepackage{stackrel}
\usepackage{dsfont}
\usepackage{amsmath}
\usepackage{enumitem}
\usepackage{amsthm}
\usepackage{fullpage}
\usepackage{xcolor}
\usepackage{cite}
\usepackage{alphalph}
\usepackage{mathrsfs}
\usepackage{romannum}
\usepackage{graphicx}
\usepackage{amsmath,lipsum}
\usepackage{amssymb}
\usepackage{csquotes}
\usepackage[colorlinks]{hyperref}
\usepackage{cleveref}
\usepackage{graphics}
\AtBeginDocument{\pagenumbering{arabic}}
\numberwithin{equation}{section}
\newcommand{\ep}{\varepsilon}

\def\d{\displaystyle}
\def\bs{\boldsymbol}
\def\ov{\overline}
\def\p{\partial}

\def\u{\boldsymbol{u}}
\def\v{\boldsymbol{v}}
\def\R{\mathbb{R}}
\def\Rn{\mathbb{R}^n}
\def\Zn{\mathbb{Z}^n}
\def\w{\boldsymbol{w}}
\def\g{\Gamma_0^{\ep}}

\def\t{\boldsymbol{\theta}}

\def\He1{( H^1_{\g}({\mathcal{O}}_{\ep}^*)) ^n}
\def\H1{(H^1({\mathcal{O}}_{\ep}^*))^n}

\def\w{\rightharpoonup}
\newtheorem{thm}{ \bf Theorem}[section]

\newtheorem{lem}[thm]{ \bf Lemma}
\newtheorem{prop}[thm]{ \bf Proposition}
\newtheorem{defn}[thm]{ \bf Definition}

\newcommand{\abs}[1]{\lvert#1\rvert}

\newcounter{cnt1}
\newcounter{cnt2}

\newcommand{\blr}{\begin{list}{$($\roman{cnt1}$)$}
		{\usecounter{cnt1} \setlength{\topsep}{0pt}
			\setlength{\itemsep}{0pt}}}
	\newcommand{\bla}{\begin{list}{$($\alph{cnt2}$)$}
			{\usecounter{cnt2} \setlength{\topsep}{0pt}
				\setlength{\itemsep}{0pt}}}

		\usepackage[symbol]{footmisc}
		\title{Optimal control problem for Stokes system: Asymptotic analysis via unfolding method in a perforated domain }
		\begin{document}
			\author{Swati Garg and Bidhan Chandra Sardar\footnotemark[1]\\
		\small{	Department of Mathematics} \\
			\small{Indian Institute of Technology Ropar}\\
		\small{	Rupnagar-140001, Punjab, India}\\
			\small{swati.19maz0006@iitrpr.ac.in, swatigargmks@gmail.com}\\ \small{bcsardar@iitrpr.ac.in, bcsardar31@gmail.com}}
			
			\maketitle
			\maketitle
			\footnotetext[1]{Corresponding author}
			
\begin{abstract}
This article's subject matter is the study of the asymptotic analysis of the optimal control problem  (OCP) constrained by the stationary Stokes equations in a periodically perforated domain.  We subject the interior region of it with distributive controls. The Stokes operator considered involves the oscillating coefficients for the state equations. We characterize the optimal control and, upon employing the method of periodic unfolding, establish the convergence of the solutions of the considered OCP to the solutions of the limit OCP governed by stationary Stokes equations over a non-perforated domain. The convergence of the cost functional is also established.
 \end{abstract}

\section{Introduction}
In this article, we consider the optimal control problem (OCP) governed by generalized stationary Stokes equations in a periodically perforated domain  $\mathcal{O}^*_{\ep}$ (see Section \ref{sec2}, on the domain description). The size of holes in the perforated domain is of the same order as that of the period, and the holes are allowed to intersect the boundary of the domain. The control is applied in the interior region of the domain, and we wish to study the asymptotic analysis (homogenization) of an interior OCP subject to the constrained stationary Stokes equations with oscillating coefficients.	
\par One can find several works in the literature regarding the homogenization of Stokes equations over a perforated domain. Using the multiple-scale expansion method, the authors in \cite{HE1975}  studied the homogenization of Stokes equations in a porous medium with the Dirichlet boundary condition on the boundary of the holes. They obtained the Darcy's law as the limit law in the homogenized medium. In \cite{CDE1996}, the authors  considered the Stokes system in a periodically perforated domain with non-homogeneous slip boundary conditions depending upon some parameter $\gamma$. Upon employing the  Tartar's method of oscillating test functions  they obtained under homogenization, the limit laws, viz., Darcy's law ( for $\gamma < 1$), Brinkmann's law (for $\gamma =1$), and Stokes's type law (for $\gamma > 1$). In \cite{Zaki2012}, the author studied a similar problem using the method of periodic unfolding in perforated domains by \cite{CDZ2006}. Further, the type of behavior as seen in \cite{CDE1996} was already observed in \cite{CC1985} by the authors while studying the homogeneous Fourier boundary conditions for the two-dimensional Stokes equation. Likewise, in \cite{A1991, A1989},  the author examined the Stokes equation in a perforated domain with holes of size much smaller than the small positive parameter $\ep$, wherein they considered the boundary conditions on the holes to be of the Dirichlet type in \cite{A1989} and the slip type in \cite{A1991}. The domain geometry, more specifically, the size of the holes, determines the kind of limit law in these works. Also, 
the author in \cite{B1996}  employed the $\Gamma-$ convergence techniques to get comparable results.
\par A few works concern the homogenization of the  OCPs governed by the elliptic systems over the periodically perforated domains with different kinds of boundary conditions on the boundary of holes (of the size of the same order as that of the period). In this regard, with the use of different techiniques, viz., $H_0-$ convergence in \cite{KS1997}, two-sclae convergence in \cite{ MN2009}, and unfolding methods in  \cite{C2016, M2022}, the homogenized OCPs were thus obtained over the non-perforated domains. Further, in context to the Stokes system, the authors in \cite{MN2008} studied the homogenization  of the OCPs subject to the Stokes equations with Dirichlet boundary conditions on the boundary of holes, where the  size of the holes is of the same order as that of the period. Here, the authors could obtain the homogenized system, pertaining only to the case when the set of admissible controls was unconstrained. For more literature concerning the homogenization of optimal control problems in perforated domains, the reader is reffered to \cite{KS1999, CDJM2016, PZ2002, DPS2022, DPS20221} and the references therein.
\par The present article introduces an interior OCP subject to the generalized stationary Stokes equations in a periodically perforated domain $\mathcal{O}^*_{\ep}$. On the boundary of holes that do not intersect the outer boundary, the homogeneous Neumann boundary condition is prescribed, while on the rest part of the boundary, the homogeneous Dirichlet boundary condition is prescribed. The underlying objective of this article is to study the homogenization of this OCP.   More specifically, we consider the minimization of the $L^2-$cost functional $(\ref{OCPE})$, which is subject to the constrained generalized stationary Stokes equations \eqref{2e2}.
\par The Stokes equations are generalized in the sense that we consider a second-order elliptic linear differential operator in divergence form with oscillating coefficients, i.e., $-\operatorname{div}\left( { A_{\ep}\nabla}\right)$, first studied for the fixed domain in \cite[Chapter 1]{BLP1978}, instead of the classical Laplacian operator. Here, the action of the  scalar operator $-\operatorname{div}\left( { A_{\ep}\nabla}\right)$ is defined in a "diagonal" manner on any vector $\bs{u} = (u_1, \dots, u_n)$, with components $u_1, \dots, u_n$ in the $H^1$ Sobolev space. That is, for $1\le i \le n,$ we have $\left(-\operatorname{div}\left( { A_{\ep}\nabla \boldsymbol{u}}\right)\right)_i = -\operatorname{div}\left( { A_{\ep}\nabla {u}_{i}}\right)$.  
The main difficulty observed during the homogenization was identifying the limit pressure terms appearing in the state and the adjoint systems, which we overcame by introducing suitable corrector functions that solved some cell problems.  We thus obtained the limit OCP associated with the stationary Stokes equation in a non-perforated domain.
\par The layout of this article is as follows: In the next section, we introduce the periodically perforated domain $\mathcal{O}_{\ep}^*$ along with the notations that will be useful in the sequel. Section \ref{Section 3} is devoted to a detailed description of the considered OCP and the derivation of the optimality condition, followed by the characterization of the optimal control.  In Section \ref{Section 5}, we  derive a priori estimates of the solutions to the considered OCP and its corresponding adjoint problem. In Section \ref{Section 4}, we recall the definition of the method of periodic unfolding in perforated domains (see, \cite{CDGO2006, CDDGZ2012}) and a few of its properties. Section \ref{Section 6}, refers to the limit (homogenized) OCP. Finally,  we derive the main convergence results in Section \ref{Section 7}. 
			
\section{Domain description and Notation }\label{sec2}
\subsection{Domain description}
Let $\{b_1, . . . , b_n\}$ be a basis of $\Rn \, (n\ge 2),$ and  $W$ be the associated reference cell defined as
\begin{equation*}
W = \left\{ w\in \Rn\, |\,  w = \sum_{i=1}^{n} w_i b_i, \, (w_1,\dots, w_n) \in (0, 1)^n \right\}.
\end{equation*}
Let us denote $\mathcal{O}$, $W,$ and $W^* = W\backslash Y
$ by an open bounded subset of $\Rn$, a compact subset of $\overline{W}$, and the perforated reference cell, respectively. It is assumed that the boundary of $Y$ is Lipschitz continuous and has a finite number of connected components.\\
Also, let $\ep > 0$ be a sequence that converges to zero and set 
\begin{equation*}
				\mathcal{T} = \left\{ \zeta \in \Rn\, |\,  \zeta = \sum_{i=1}^{n} z_i b_i, \, (z_1, \dots , z_n) \in \Zn \right\}, \quad  	\mathcal{Z}_{\ep} =  \left\{ \zeta \in \mathcal{T} \, |\, \ep (\zeta + W) \subset \mathcal{O}\right\}.
\end{equation*}
We take into account the perforated domain $\mathcal{O}_\ep^*$ (see Figure \ref{s1}) given by
$
			\mathcal{O}_\ep^* = \mathcal{O} \backslash Y_{\ep}, 
$
where $ { Y_{\ep} = \cup_{\zeta \in \mathcal{T}}\, \ep(\zeta + Y)}$.
\begin{figure}[ht!]
				\centering
				\includegraphics[ width=90mm]{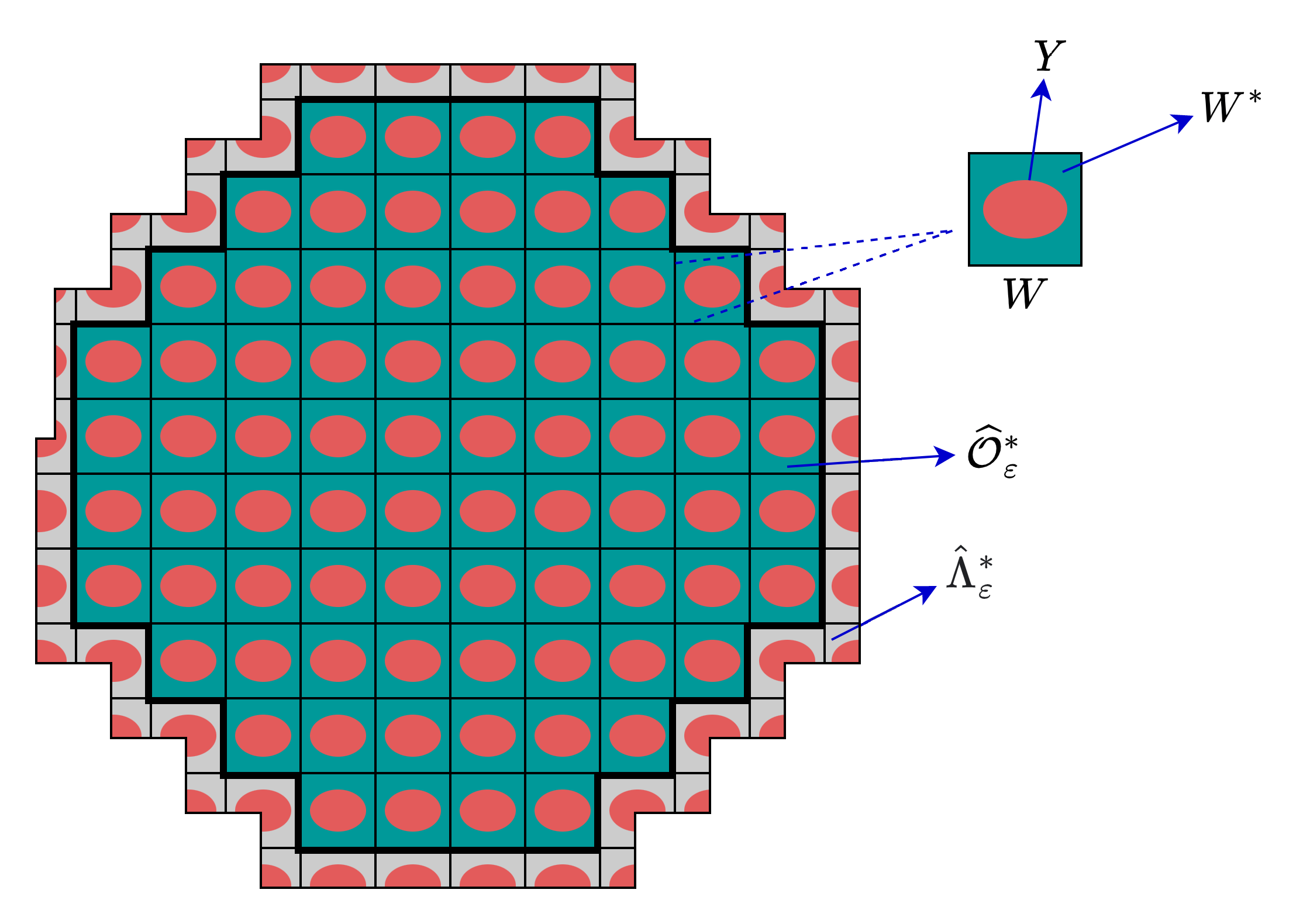}
				\caption{ The Perforated domain $\mathcal{O}_\ep^*$ and the reference cell W. \label{s1}}
\end{figure}
Now, let us denote  $\widehat{\mathcal{O}}_{\ep}$ as the interior of the largest union of $\ep(\zeta + \overline{W})$	cells such that $\ep(\zeta + {W}) \subset {\mathcal{O}}$, while  $\Lambda_{\ep} \subset {\mathcal{O}}$ as containing the parts from $\ep(\zeta + \overline{W})$ cells intersecting the boundary $\p{\mathcal{O}}$. More precisely, we write $\Lambda_{\ep} = {\mathcal{O}} \backslash \widehat{\mathcal{O}}_{\ep}$, where
\begin{equation*}
				\widehat{\mathcal{O}}_{\ep} = interior \left\{ \cup_{\zeta \in \mathcal{Z}_{\ep}} \, \ep(\zeta + \overline{W}) \right\}.
\end{equation*}
The associated perforated domains are defined as 
\begin{equation*}
				\widehat{\mathcal{O}}_{\ep}^* = \hat{\mathcal{O}}_{\ep} \backslash Y_{\ep}, \quad \hat{\Lambda}_{\ep}^* = {\mathcal{O}}_{\ep}^* \backslash \widehat{\mathcal{O}}_{\ep}^*.
\end{equation*}
Also, we denote the boundary of the perforated domain ${\mathcal{O}}_{\ep}^*$ as
\begin{equation*}
				\p{\mathcal{O}}_{\ep}^* = \Gamma_1^{\ep} \cup \Gamma_0^{\ep}, \quad \text{where}\,  \Gamma_1^{\ep} = \p\widehat{\mathcal{O}}_{\ep} \cap \p Y_{\ep} \, \text{and} \,  \Gamma_0^{\ep} = \p{\mathcal{O}}_{\ep}^* \backslash \Gamma_1^{\ep},
\end{equation*}
which means that $\Gamma_1^{\ep}$ denotes the boundary of set of holes contained in $\widehat{\mathcal{O}}_{\ep}$.\\
In Figure \ref{s1}, $\widehat{\mathcal{O}}_{\ep}^*$ and $\hat{ \Lambda}^*_{\ep}$ respectively represent the dark perforated part and the remaining part of the perforated domain ${\mathcal{O}}_{\ep}^*$. While, $\Gamma_1^{\ep}$ and $\Gamma_0^{\ep}$ respectively represent the boundary of holes contained in $\widehat{\mathcal{O}}_{\ep}^*$ and the boundary of holes contained in $\hat{\Lambda}^*_{\ep}$ along with the outer boundary $\p{\mathcal{O}}$. In the following, we introduce a few notations that we shall use throughout this article.
			
\subsection{Notation}
\begin{itemize}
	            \item $A_{\ep}(x) = A(\frac{x}{\ep})$ a.e. in ${\mathcal{O}}$, for all $\ep > 0$.
				\item $\bs v_{\bs\ep} = (v_{\ep1},\dots, v_{\ep n}),$ for any bold symbol vector function $\bs v_{\ep}$.
				\item $\bs v = (v_1, \dots, v_n),$ for any bold symbol vector function $\bs v$.
				\item $\bs \eta_{\bs\ep}$ denotes the  outward normal unit vector to $\Gamma^{\ep}_1$.
				\item $\bs  \eta $ denotes the  outward normal unit vector to $\p {\mathcal{O}}$.
				\item $M^t$ denotes the transpose of any matrix $M$.
				\item $\widetilde{\psi}$ is the zero extension  of any function $\psi$ outside ${\mathcal{O}}_{\ep}^*$ to the whole of ${\mathcal{O}}$.
				\item $\widetilde{\bs\psi} = (\widetilde{\psi_1}, \cdots, \widetilde{\psi_n})$, for any vector function $\bs\psi$.
				\item $|F|$ is the Lebesgue measure of the measurable set $F$.
				\item $\Theta = \frac{|W^*|}{|W|}$, the proportion of the perforated reference cell $W^*$ in the reference cell $W$.
				\item $\mathcal{M}_{W^*}(\phi)$ is the mean value of $\phi$ on the perforated reference cell $W^*$.
				\item $\mathcal{M}_{W^*}(\bs\phi)= (\mathcal{M}_{W^*}(\phi_1), \cdots, \mathcal{M}_{W^*}(\phi_n))$, for vector function $\bs \phi$.
				\item $\{D\to\mathbb{R}\}$, the set of all real valued functions defined on domain $D$.
				\item $\mathcal{D}(\Omega)$, is the space of infinitely many times differentiable functions with compact support in $\Omega$, for any open set $\Omega \in \R^n$.
\end{itemize}
			
\section{Problem description and Optimality condition}\label{Section 3}
Let us consider the following OCP associated with Stokes system:
\begin{equation}\label{OCPE}
				       \inf_{{\bs{\t}_{\bs\ep}} \in (L^2({\mathcal{O}}_{\varepsilon}^*))^n} \left\lbrace J_{\ep}(\boldsymbol{\t}_{\bs\ep})=\frac{1}{2} \int_{{\mathcal{O}}_{\ep}^*} \abs{\u_{\bs\ep}(\t_{\bs\ep})-\u_{\bs d}}^{2} +\frac{\tau}{2}\int_{{\mathcal{O}}_{\ep}^*} \abs{\boldsymbol{\t}_{\bs\ep}}^{2}\right\rbrace,
\end{equation}	
subject to 
\begin{equation}\label{2e2}
			       	\left\{
			     	\begin{array}{rlc}
					  -\displaystyle \operatorname{div}\left( { A_{\ep}\nabla \boldsymbol{u_\ep}}\right)   + \nabla p_{\ep}
					  & ={\boldsymbol{\t}_{\bs\ep} } & \text{in }{\mathcal{O}}_{\ep}^*,\\[2mm]
					 \operatorname{div}({\boldsymbol{u_\ep}}) & =0 &
					 \text{in } {\mathcal{O}}_{\ep}^*,\\[2mm]
					 \d \bs{\eta_{\ep}} \cdot A_{\ep}\nabla \boldsymbol{u_\ep}  -{p}_{\ep} \bs{\eta_{\ep}} & =  \boldsymbol{0}  & \text{on } \Gamma^{\ep}_1,\\
					  \boldsymbol{u_\ep} & = \boldsymbol{0} & \text{on }  \Gamma^{\ep}_0,
				\end{array}
				\right.
\end{equation}
where the desired state $\u_{\bs d}=(u_{d_1}, \dots, u_{d_n})$ is defined on the space $(L^{2}({\mathcal{O}}))^n$, $\boldsymbol{\t}_{\bs\ep} $ is a control function defined on the space $\left( L^2({\mathcal{O}}_{\ep}^*)\right)^n$    and $\tau >0$ is a given regularization parameter. 
Here, the matrix $A_{\ep}(x) = A (\frac{x}{\ep})$, where $A (x) =\left( a_{ij}(x)\right)_{1\le i,j\le n}$ defined on the space $\left(  L^{\infty}({\mathcal{O}}) \right)^{n \times n} $ is assumed  to obey the uniform ellipticity condition: there exist real constants $m_1,\,m_2>0$ such that $  m_1 ||{\lambda}||^2 \le  \sum_{i, j = 1}^{n} a_{ij}(x)\lambda_{i} \lambda_{j} \le m_2 ||{\lambda}||^2 $ for all $\lambda \in \mathbb{R}^n$, which is endowed with an Eucledian norm denoted by $||\cdot||$.  Also, we understand the action of scalar boundary operator   $\bs{\eta_{\ep}} \cdot A_{\ep}\nabla $ on the vector $\boldsymbol{u}_{\bs\ep}|_{\Gamma^{\ep}_1}$ in a ''diagonal'' manner:  $\left(  \bs{\eta_{\ep}} \cdot A_{\ep}\nabla \boldsymbol{u}_{\bs\ep} \right)_i = \bs{\eta_{\ep}} \cdot A_{\ep}\nabla {u}_{\ep i}$,  for $1\le i \le n$.	\\[1mm]
We introduce the function space $ (H^1_{\g}({\mathcal{O}}_{\ep}^*))^n := \{{\bs \phi} \in (H^1({\mathcal{O}}_{\ep}^*))^n\  | \   \bs{\phi}|_{\g} = \bs0\}$. This is a Banach space endowed with the norm
\begin{equation*}
				|| \bs{\phi}||_{(H^1_{\g}({\mathcal{O}}_{\ep}^*))^n}  := 	|| \nabla{\bs{\phi}}||_{(L^2 ({\mathcal{O}}_{\ep}^*))^{n \times n}}, \quad \forall \bs{\phi} \in (H^1_{\g}({\mathcal{O}}_{\ep}^*))^n.
\end{equation*}
{\bf{Definition 2.1.}}
We say a pair $(\u_{\bs\ep}, p_\ep) \in (H^1_{\g}({\mathcal{O}}_{\ep}^*))^n \times L^{2}({\mathcal{O}}_{\ep}^*) $ is a weak solution to \eqref{2e2} if, for all $\bs{\phi} \in (H^1_{\g}({\mathcal{O}}_{\ep}^*))^n$, 
			
\begin{equation}\label{2e3}
				\int_{{\mathcal{O}}_\ep^*} A_{\ep} \nabla{\bs {u}_{\bs\ep}} \colon \nabla{\bs \phi} \, dx - \int_{{\mathcal{O}}_\ep^*} p_\ep \operatorname{div}(\bs{\phi})\, dx = \int_{{\mathcal{O}}^*_\ep}  \boldsymbol {\theta}_{\bs\ep} \cdot \bs{\phi}\, dx,
\end{equation}
and, for all $w \in L^{2}({\mathcal{O}}_{\ep}^*),$ \\
\begin{equation}\label{2e4}
				\int_{{\mathcal{O}}_\ep^*} \operatorname{div}(\bs{u_\ep})\ w\ dx = 0.
\end{equation} 

Here, $(\colon)$ and $(\cdot)$ represent the summation of the component-wise multiplication of the matrix entries and the usual scalar product of vectors, respectively. The existence of a unique weak solution	$(\u_{\bs\ep}({\boldsymbol{\t}_{\bs\ep}}),p_{\ep}) \in  (H^1_{\g}({\mathcal{O}}_{\ep}^*))^n\times L^2({\mathcal{O}}_{\ep}^*) $  of the system \eqref{2e2} follows analogous to \cite[Theorem \Romannum{4}.7.1]{FP2013}. Also, for each $ \ep > 0$, there exists a unique solution to the problem \eqref{OCPE} that can be proved along the same lines as in \cite[Chapter 2, Theorem 1.2]{Lions1971}. We call the optimal solution to \eqref{OCPE} by  the triplet $(\bs{\overline{u}_{\bs\ep}},\ov{p}_{\ep}, \ov{\bs{\t}}_{\bs\ep} ) $, with $\bs{\overline{u}_{\bs\ep}}$, $\ov{p}_{\ep}$, and $\ov{\bs{\t}}_{\bs\ep}$ as optimal state, pressure, and control, respectively.\\
			
\noindent \textbf{Optimality Condition}: The optimality condition is given by $  J^{\prime}_{\ep}(\t)\cdot (\t - \ov {\t}_{\bs\ep}) \ge 0$, for all ${\bs{\t}} \in (L^2({\mathcal{O}}_{\varepsilon}^*))^n$ (see,  \cite[Chapter 2, Page 48]{Lions1971}). One can obtain the further simplification of this condition as  $\int_{{\mathcal{O}}^*_\ep} (\ov{\bs{v}}_{\bs\ep} + \tau\, \ov\t_{\bs\ep} ) \cdot(\t - \ov {\t}_{\bs\ep}) \ge 0$,  for all ${\bs{\t}} \in (L^2({\mathcal{O}}_{\varepsilon}^*))^n$ (see,  \cite[Chapter 2]{Lions1971}), where the pair $(\ov{\bs{v}}_{\ep}, \ov{q}_{\ep})$ is the solution to the following adjoint problem:
			
\begin{equation}\label{4e1}
				\left\{
				\begin{array}{rlc}
					-\displaystyle \operatorname{div}\left( { A^t_{\ep}\nabla \boldsymbol{\ov v_\ep}}\right)  + \nabla \ov{q}_{\ep}
					& ={\boldsymbol{\ov u}_{\bs\ep} - \bs{u_d}} & \text{in }{\mathcal{O}}_{\ep}^*,\\[2mm]
					\operatorname{div}({\boldsymbol{\ov v_\ep}}) & =0 &
					\text{in } {\mathcal{O}}_{\ep}^*,\\[2mm]
					\d\bs{\eta_{\ep}} \cdot {A^t_{\ep}}\nabla \boldsymbol{\ov v_\ep}  -{\ov q}_{\ep} \bs{\eta}_{\bs\ep} & = \boldsymbol{0}  & \text{on } \Gamma^{\ep}_1,\\
					\boldsymbol{\ov v_\ep} & = \boldsymbol{0} & \text{on }  \Gamma^{\ep}_0.
				\end{array}
				\right.
\end{equation}
We call $\boldsymbol{\ov v_\ep}$ and ${\ov q}_{\ep}$, the adjoint state and pressure, respectively. The existence of unique weak solution $(\boldsymbol{\ov v_\ep}, {\ov q}_{\ep})$ to \eqref{4e1} can now be proved in a way  similar to that of system \eqref{2e2}.\\[2mm]
The following theorem characterizes the optimal control, the proof of which follows analogous to standard procedure laid in \cite[Chapter 2, Theorem 1.4]{Lions1971}.
\begin{thm}\label{4t1}
Let $\left(\overline{\bs{u}}_{\bs\ep},\overline{{p}}_{\ep}, \overline{\boldsymbol{\t}}_{\bs\ep}\right) $ be the optimal solution of the problem \eqref{OCPE} and $(\bs{\overline{v}_{\bs\ep}}, \ov{q}_{\ep})$ solves
\eqref{4e1}, then the optimal control  is characterized by\\
\begin{equation}\label{4e2}
					\overline{\boldsymbol{\t}}_{\bs\ep}  = -\frac{1}{\tau} \bs{\overline{v}}_{\bs\ep}\text {  a.e. in } {\mathcal{O}}_{\ep}^*.
\end{equation}
Conversely, suppose that a triplet $ (\boldsymbol {\check{u}}_{\bs\ep}, \check{p}_{\ep}, \boldsymbol{\check{\t}}_{\bs\ep})
				\in \left( H^1_{\g}({\mathcal{O}}_{\ep}^*)\right) ^n \times  L^2({\mathcal{O}}_{\ep}^*) \times  \left( L^2({{\mathcal{O}}_{\ep}^*})\right) ^n$ and a pair $ (\boldsymbol {\check{v}}_{\bs\ep}, \check{q}_{\ep}) \in \left( H^1_{\g}({\mathcal{O}}_{\ep}^*)\right) ^n \times L^2({\mathcal{O}}_{\ep}^*)$ solves the following system:
\begin{equation*}\label{4e3}
					\left\{
					\begin{array}{rlc}
						&-\displaystyle \operatorname{div}\left( { A_{\ep}\nabla \boldsymbol{\check u_\ep}}\right)  + \nabla \check p_{\ep}
						= -\frac{1}{\tau} \bs{\check{v}_{\ep}} & \text{in }{\mathcal{O}}_{\ep}^*,\\[2mm]
						&-\displaystyle \operatorname{div}\left( { A^t_{\ep}\nabla \boldsymbol{\check v_\ep}}\right)  + \nabla \check{q}_{\ep}
						={\boldsymbol{\check u_{\ep}} - \bs{u_d}} & \text{in }{\mathcal{O}}_{\ep}^*,\\[2mm]
						&\d \operatorname{div}({\boldsymbol{\check{u}_\ep}}) =0 ,\, 	\operatorname{div}({\boldsymbol{\check{v}_\ep}})=0 &	\text{in } {\mathcal{O}}_{\ep}^*,\\[2mm]
						&\d \bs{\eta_{\ep}} \cdot A_{\ep}\nabla \boldsymbol{\check u_\ep}  -\check{p}_{\ep} \bs{\eta_{\ep}}  = \boldsymbol{0}  & \text{on } \Gamma^{\ep}_1,\\[2mm]
						&\d \bs{\eta_{\ep}} \cdot {A^t_{\ep}}\nabla \boldsymbol{\check v_\ep} -{\check q}_{\ep} \bs{\eta_{\ep}}  = \boldsymbol{0}  & \text{on } \Gamma^{\ep}_1,\\
						&\boldsymbol{\check{v}_\ep}  = \boldsymbol{0},\,
						\boldsymbol{\check{u}_\ep} = \boldsymbol{0} &\text{on } \g.
					\end{array}
					\right.
\end{equation*}
Then the triplet $(\check{\bs{u}}_{\bs\ep},\check{p}_\ep,  -\frac{1}{\tau} \bs{\check{v}}_{\bs\ep})$\ is the optimal solution of \eqref{OCPE}.
\end{thm}

\section{A priori estimates}\label{Section 5}
This section concerns  the derivation of estimates for the optimal solution to the problem \eqref{OCPE} and the associated solution to the adjoint problem \eqref{4e1}. These estimates are uniform and independent of the parameter $\ep$. Towards attaining this aim, we first evoke the following two lemmas:
\begin{lem}[Lemma A.4, {\cite {AM1993}}]\label{4l2} There exists a constant $C \in \mathbb{R}^+$, independent of $\ep$,  such that
\begin{equation*}
		||{\bs v}||_{ L^2({\mathcal{O}}_{\ep}^*)^n} \le C ||\nabla {\bs v}||_{\left( L^2({\mathcal{O}}_{\ep}^*)\right) ^{n\times n}}, \quad \forall\ {\bs v} \in \He1.
\end{equation*}
\end{lem}
\begin{lem}[Lemma 5.1, {\cite {CC1985}}]\label{4s3} For each $\ep>0$ and $q_{\ep} \in L^2({\mathcal{O}}_{\ep}^*)$, there exists  $\bs{g}_{\bs\ep} \in \He1 $ and a constant $C  \in \mathbb{R}^+$, independent of $\ep$, such that
\begin{equation} \label{5.1}
		\operatorname{div}(\bs{g}_{\bs\ep})= q_{\ep}\ \ \text{and}\ \ ||\nabla {\bs g}_{\bs\ep}||_{\left( L^2({\mathcal{O}}_{\ep}^*)\right)^{n\times n}} \le C({\mathcal{O}})\ ||q_{\ep}||_{L^2({\mathcal{O}}_{\ep}^*)}.	
\end{equation}
\end{lem}
\begin{thm}\label{4t4}
For each $\ep > 0,$ let $\left(\overline{\bs{u}}_{\bs\ep},\overline{{p}}_{\ep}, \overline{\boldsymbol{\t}}_{\bs\ep}\right) $ be the optimal solution of the problem \eqref{OCPE} and $(\bs{\overline{v}_{\ep}}, \overline{q}_{\ep} )$ solves the corresponding adjoint problem $(\ref{4e1})$. Then, one has $\ov{\bs{\t}}_{\bs\ep} \in (H^1_{\g}({\mathcal{O}}_{\ep}^*))^n $ and there exists a constant $C  \in \mathbb{R}^+$, independent of $\ep$ such that 
\begin{align}
		&\left\|\bar{\bs \theta}_{\bs\varepsilon}\right\|_{\left(L^{2}\left({\mathcal{O}}_{\varepsilon}^{*}\right)\right)^n} \leq C, \label{5.2}\\
		&\left\|\bar{\bs u}_{\bs\varepsilon}\right\|_{\He1} \leq C,\label{5.3} \\
		&\left\|\bar{\bs v}_{\bs \varepsilon}\right\|_{\He1} \leq C,\label{5.4} \\
		&\left\|\bar{p}_{\varepsilon}\right\|_{L^2({\mathcal{O}}_{\ep}^*)} \leq C,\label{5.5}\\
		&\left\|\bar{q}_{\varepsilon}\right\|_{L^2({\mathcal{O}}_{\ep}^*)} \leq C.\label{5.6}
\end{align}
\end{thm}
\begin{proof}
Let {\color{black}{ $\bs{u}_{\bs\ep} (\bs 0)$ denotes the solution}} to \eqref{2e2} corresponding to $\bs{\t}_{\bs\ep}=\bs 0$. In view of Lemma \ref{4l2}, one can show that $\Vert {\bs{u}}_{\bs\ep} (\bs0)\Vert_{(L^{2}({\mathcal{O}}_{\ep}^{*}))^n} \le 0$, i.e., $ {\bs{u}}_{\bs\ep} (\bs0) = \bs0$  in ${(L^{2}({\mathcal{O}}_{\ep}^{*}))^n} $. Using this and the optimality of solution  $(\ov{\bs{u}}_{\bs\ep},\ov{{p}}_{\ep}, \ov{\bs{\t}}_{\bs\ep})$ to problem \eqref{OCPE}, we have
\begin{equation*}
		\Vert \ov{\bs{u}}_{\bs\ep} (\ov{\bs{\t}}) - \bs{u_d} \Vert^2_{(L^2({\mathcal{O}}_{\ep}^*))^n} + \tau \Vert \ov{\bs{\t}}_{\bs\ep} \Vert^2_{(L^{2}({\mathcal{O}}_{\ep}^{*}))^n} \le 	\Vert {\bs{u}}_{\bs\ep} (\bs0) - \bs{u_d} \Vert^2_{(L^2({\mathcal{O}}_{\ep}^*))^n} \le C,
\end{equation*}
which gives estimate \eqref{5.2}. Now, let us take $\ov{\bs u}_{\bs\ep}$ as a test function in $(\ref{2e3})$. Considering \eqref{5.2} and the uniform ellipticity condition of matrix $A_{\ep}$, one obtains upon applying the Cauchy-Schwarz inequality along with the  Lemma \ref{4l2}, the following:
\begin{equation*}
		m_1 \Vert\nabla {\ov{\bs{u}}}_{\bs\ep}\Vert^2_{\left( L^2({\mathcal{O}}_{\ep}^*)\right) ^{n\times n}} \le  \int_{{\mathcal{O}}_{\ep}^*} A_{\ep} \nabla{{\ov{\bs{u}}}_{\bs\ep}} \colon  \nabla{{\ov{\bs{u}}}_{\bs\ep}}\ dx \le C\ \Vert {\ov{\bs{\t}}_{\bs\ep}} \Vert _{\left( L^2({\mathcal{O}}^*_{\ep})\right) ^n} \Vert\nabla {\ov{\bs{u}}}_{\bs\ep}\Vert_{\left( L^2({\mathcal{O}}_{\ep}^*)\right) ^{n\times n}},
\end{equation*}
from which estimate  \eqref{5.3} follows.\\[1mm]
Owing to Lemma \ref{4s3}, for given $\ov{p}_{\ep} \in L^2({\mathcal{O}}_{\ep}^*)$,  there exists $\bs{g}_{\ep} \in \He1 $ satisying  $\operatorname{div}(\bs{g}_{\bs\ep})= \ov p_{\ep}$. Corresponding to $\ov{\bs{\t}}_{\bs\ep}$, taking $\bs v = \bs{g}_{\bs\ep}$  in \eqref{2e3}, we get
\begin{equation}\label{4e17}
		\Vert \ov{p}_{\ep}\Vert^2_{L^2({\mathcal{O}}_{\ep}^*)}= 	\int_{{\mathcal{O}}_{\ep}^*} A_{\ep} \nabla{\bs{\ov {u}}}_{\bs\ep} \colon  \nabla{\bs{g}_{\bs\ep}}\ dx- \int_{{\mathcal{O}}^*_\ep}  \ov{\boldsymbol {\theta}}_{\bs\ep} \cdot \bs{g}_{\bs\ep}\, dx.
\end{equation}
In view of \eqref{5.1}, \eqref{5.2} and \eqref{5.3}, and the uniform ellipticity condition of the  matrix $A_{\ep}$, one obtains from (\ref{4e17}) upon employing the Cauchy-Schwarz inequality and Lemma \ref{4l2}, the following: 
\begin{equation*}
		\begin{aligned}
			\Vert \ov{p}_{\ep}\Vert^2_{L^2({\mathcal{O}}_{\ep}^*)} &\le  \left( m_2 \Vert\nabla \ov{\bs u}_{\bs\ep}\Vert_{\left( L^2({\mathcal{O}}_{\ep}^*)\right)^{n\times n}} + C \Vert \ov {\bs \theta}_{\bs\ep}\Vert_{\left( L^2({\mathcal{O}}^*_{\ep})\right) ^n}  \right)  \Vert\nabla {\bs g}_{\bs\ep}\Vert_{\left( L^2({\mathcal{O}}_{\ep}^*)\right)^{n\times n}},	
		\end{aligned}
\end{equation*}
which gives the estimate \eqref{5.5}. Likewise, one can easily obtain the estimates \eqref{5.4} and \eqref{5.6} following the above discussion. Finally, from  \eqref{4e2}, we obtain that $\ov{\bs{\t}}_{\bs\ep} \in (H^1_{\g}({\mathcal{O}}_{\ep}^*))^n $.
\end{proof}

\section{The method of periodic unfolding for perforated domains}\label{Section 4}
We evokes the definition of the periodic unfolding operator and few of its properties as stated in \cite{CDGO2006, CDDGZ2012 }. 
Given $x \in \mathbb{R}^n$, we denote the greatest integer and the fractional parts of $x$ respectively by $[x]_{W}$  and $\{x\}_{W}$. That is, $[x]_{W}=\sum_{j=1}^{n} k_{j} b_{j}$ be the unique integer combination of periods and $\{x\}_{W}=x-[x]_{W}$. In particular, we have for $\varepsilon > 0$,
			$$
			x=\varepsilon\left(\left[\frac{x}{\varepsilon}\right]_{W}+\left\{\frac{x}{\varepsilon}\right\}_{W}\right), \quad \forall\,  x \in \mathbb{R}^{n}.
			$$
\begin{defn}\label{3d1}
The unfolding operator $T^*_\ep : \{{\mathcal{O}}_\ep^*\to\mathbb{R}\}\to\{{\mathcal{O}} \times W^*\to\mathbb{R}\}$ is defined as
\begin{equation*}
					T^*_\ep\left(  u \right) (x,y)  =\left\{\begin{array}{lcl}
						u\left( \ep \left[\frac{x_1}{\ep} \right]_{W} + \ep y\right) & \text{a.e.}& \text{for}\ (x,y) \in \widehat{{\mathcal{O}}}_{\ep} \times W^*, \\
						0& \text{a.e.}& \text{for}\ (x,y) \in {\Lambda}_{\ep} \times W^*.
					\end{array}
					\right. 
				\end{equation*}
\end{defn}
\noindent{Also, for any domain $D \supseteq {\mathcal{O}}_{\ep}^* $ and vector $\bs u = (u_1, \cdots, u_n) \in (\{D\to\mathbb{R}\})^n$, we define its unfolding by}
\begin{equation*}
	T_{\ep}^* (\bs u) := (T_{\ep}^* (u_1), \cdots, T_{\ep}^* (u_n)).
\end{equation*}
		
\begin{prop}\label{3p2} {\color{black} In the following there are }the properties of the unfolding operator:
\begin{enumerate}[label=(\roman*)]
					\item \label{p1}$T^*_{\ep}$ is linear and continuous from $ L^2 ({\mathcal{O}}_{\ep}^*)$  to $ L^2 ({\mathcal{O}} \times W^*)$.\label{3p21}
					\item Let  $ u, v  \in L^2({\mathcal{O}}_\ep^*)$. Then  $T^*_\ep \left(u v\right)=T^*_\ep \left(u\right) T^*_\ep \left(v\right).$\label{3p22}
					\item $\ \text{Let}\ u \in L^{2}\left({\mathcal{O}}\right).$ Then $ T^*_\ep  (u) \rightarrow u\ \text{strongly in}\ L^{2}\left({\mathcal{O}} \times W^*\right).\label{3p23}$	
					\item	$\ \text{Let}\ u \in L^{1}\left({\mathcal{O}}_\ep^*\right).\ \text{Then}$
					$$\int_{\widehat{{\mathcal{O}}}_\ep^*} u(x)\ d x =\int_{{\mathcal{O}}_\ep^*} u(x)\ d x - \int_{\hat{\Lambda}_\ep^*} u(x)\ d x  = \frac{1}{|W^*|}\int_{{\mathcal{O}} \times W^*} T^*_\ep  (u) (x,y)\ d x dy.\label{3p24}$$
					\item $ \ \text{For each}\ \ep>0,\, \text{let}\ \left\{u_{\varepsilon}\right\} \in L^{2}\left({\mathcal{O}}\right)\ \text{and}\   u_{\varepsilon}\to u\  \text{strongly in}\ L^{2}\left({\mathcal{O}} \right).$\\[2mm]
					{Then}\ 
					$			\displaystyle	T^*_\ep  (u_{\varepsilon})\to u\  \text{strongly in}\ L^{2}\left({\mathcal{O}} \times W^*\right).\label{3p25}$ 
					\item Let $v \in L^{2}\left(W^{*}\right)$ be a $W$-periodic function and $v_{\ep}(x)=v\left(\frac{x}{\ep}\right)$. Then,\label{3p26}
					\begin{equation*}
						{T}_{\ep}^{*}\left(v_{\ep}\right)(x, y) =
						\left\{\begin{array}{lcl}
					 v(y)& \text { a.e.}\ \text{for}\ (x,y) \in  \widehat{{\mathcal{O}}}_{\ep} \times W^{*},\\
						 0 &\text { a.e.}\ \text{for}\ (x,y) \in {{{\Lambda}}}_{\ep} \times W^{*}.
						\end{array}
						\right.
					\end{equation*}
				
					\item Let $f_{\ep} \in L^{2}\left({\mathcal{O}}_{\ep}^{*}\right)$ be uniformly bounded. Then, there exists $f \in  L^2({\mathcal{O}} \times W^*)$ such that ${T}_{\ep}^{*}\left(f_{\ep}\right)  \w f \, \text{weakly in}\, L^2({\mathcal{O}} \times W^*)$, and
					$$
					\widetilde{f_{\ep}} \w \frac{1}{|W|} \int_{W^*} f(\cdot, y)\, dy \text { weakly in } L^2({{\mathcal{O}}}). \label{3p27}
					$$
\end{enumerate}
\end{prop}

\begin{prop}\label{4p2}
Let ${\mathcal{O}}\subset \mathbb{R}^{n}$ be bounded with Lipschitz boundary. Let $f_{\ep} \in H^{1}({\mathcal{O}}_{\ep}^{*})$ be such that  $f_{\ep} = 0 $ on $  \partial {\mathcal{O}} \cap \partial {\mathcal{O}}_{\ep}^{*}$ and satisfy,
$$
				\left\|\nabla f_{\ep}\right\|_ {\left(L^{2}\left({\mathcal{O}}_{\ep}^{*}\right)\right)^n }  \leq C \footnote[4]{The symbol $C$ represents a generic constant that is positive and independent of $\ep$.} .
$$
Then, there exists $f \in H_{0}^{1}({\mathcal{O}})$ and $\widehat{f} \in L^{2}\left({\mathcal{O}} ; H_{p e r}^{1}\left(W^{*}\right)\right)$ with $\mathcal{M}_{W^*}(\widehat{f})=0$, such that up to a subsequence,
$$\begin{cases} & {T}_{\varepsilon}^{*}\left(\nabla f_{\varepsilon}\right) \w \nabla f +\nabla_{y} \widehat{f} \quad \text { weakly in } \left(L^{2}\left({\mathcal{O}} \times W^{*}\right)\right)	^n, \\  & {T}_{\varepsilon}^{*}\left(f_{\varepsilon}\right)\to f \quad\text { strongly in }   L^{2}\left({\mathcal{O}} ; H^{1}\left(W^{*}\right)\right). \end{cases}
$$
\end{prop}
	
\section{Limit optimal control problem}\label{Section 6}
This section presents the  limit (homogenized) system   corresponding to the problem \eqref{OCPE}, which we considered in the beginning. 
\par Let us consider the  function space
\begin{equation*}
				\left( {{H}_0^1}({\mathcal{O}})\right) ^n := \left\{{\bs {\varphi}} \in {(H^{1}({\mathcal{O}}))^n}\, |\           \boldsymbol{\varphi}|_{\p {\mathcal{O}}} = \bs0\right\},
\end{equation*}
which is a Hilbert space for the norm
\begin{equation*}
				\Vert \bs {\varphi} \Vert_{( H_0^1({\mathcal{O}}))^n} := \Vert \nabla \bs {\varphi} \Vert_{( L^2({\mathcal{O}}))^{n \times n}} \quad \forall \, \bs {\varphi} \in  ( H_0^1({\mathcal{O}}))^n.
\end{equation*} 
\par We now consider the limit OCP associated with the Stokes system
\begin{equation}\label{OCP}
				\inf_{{\bs{\t}} \in (L^2({\mathcal{O}}))^n}\left\{	J(\boldsymbol{\t})=\frac{\Theta}{2} \int_{{\mathcal{O}}}  \abs{\bs{u}- \bs{u_d}}^2\, dx
				+\frac{ \tau \Theta}{2} \int_{{\mathcal{O}}}  |\boldsymbol{\t}|^{2}\, dx \right\}, 
\end{equation}
subject to
\begin{equation}\label{6.1}
				\left\{
				\begin{array}{rlc}
					-\displaystyle \sum\limits_{j, \alpha, \beta =1}^{n} \frac{\partial }{\partial x_{\alpha}} \left( b^{\alpha \beta}_{ij}\ \frac{\partial {{u}_j} }{\partial x_{\beta}} \right)  +\nabla p &=\t \quad \text { in }\ {\mathcal{O}}, \\[2mm]
					\operatorname{div}\left(\boldsymbol{u}\right) &=0 \quad \text { in } {\mathcal{O}}, \\
					\boldsymbol{u} &= \bs 0  \quad \text { on } \p{\mathcal{O}},
				\end{array}
				\right.
\end{equation}
where the tensor $B = (b^{\alpha \beta}_{ij}) = (b^{\alpha \beta}_{ij})_{1\le i, j, \alpha, \beta \le n}$ is constant, elliptic, and  for $1\le i, j, \alpha, \beta \le n$, is given by
\begin{equation*}\label{6.2}
				b^{\alpha \beta}_{ij} = a^{\alpha \beta}_{ij} - \frac{1}{|W^*|} \int_{W^*}	A(y) \nabla_{y}\left(\bs P_j^{\beta} -\bs \chi_j^{\beta}\right)  \colon \nabla_{y} \bs \chi_i^{\alpha}\, dy,
\end{equation*}
with $  a^{\alpha \beta}_{ij} = \frac{1}{|W^*|} \int_{W^*}	A(y) \nabla_{y}\left(\bs P_j^{\beta} -\bs \chi_j^{\beta}\right)  \colon \nabla_{y} P_i^{\alpha}\, dy$ as the entries of the constant tensor $A^0$, $\bs P_j^{\beta} = \bs P_j^{\beta}(y) = (0, \dots, y_j, \dots, 0)$ with $y_j$ at the $\beta$-th position, and for $1\le j, \beta \le n$, the correctors $ (\bs \chi_j^{\beta}, \Pi_j^{\beta}) \in (H^1(W^*))^n \times  L^2(W^*)$ solves the cell problem
\begin{equation}\label{6.3}
				\left\{
				\begin{array}{rlc}
					-\displaystyle \operatorname{div}_{y}\left(  A(y)\nabla_{y} (\bs P_j^{\beta} - \bs \chi_j^{\beta} )\right)  + \nabla_{y} \Pi_j^{\beta}
					&= \bs 0 \quad \text{in } W^*,\\[2mm]
					\displaystyle \bs{\eta} \cdot A(y)\nabla_{y} (\bs P_j^{\beta} - \bs \chi_j^{\beta} )  - \Pi_j^{\beta}  \bs \eta & = \bs 0  \quad \text{on } \p W^*\backslash\p W,\\[2mm]
					\d \operatorname{div}_{y} ( \bs P_j^{\beta} - \bs \chi_j^{\beta})  &=0 \quad	\text{in } W^*,\\[2mm]
					\d (\bs \chi_j^{\beta}, \Pi_j^{\beta})  & W^* \text{- periodic}, \\[2mm]
					\d \mathcal{M}_{W^*} (\bs \chi_j^{\beta})  &= \bs0.
				\end{array}
				\right.
\end{equation}
\par 
The existence of this unique pair $(\bs u, p) \in (H_0^1 ({\mathcal{O}}))^n \times L^2({\mathcal{O}})$ can be found in \cite[Chapter 1]{BLP1978}. Further, the problem \eqref{OCP} is a standard one and there exists a unique weak solution to it, one can follow the arguments introduced in \cite[Chapter 2, Theorem 1.2]{Lions1971}. We call the triplet $(\ov{\bs{u}}, \overline{p} , \ov{\bs{ \t}}) \in 	(H_0^1 ({\mathcal{O}}))^n \times L^2({\mathcal{O}}) \times ( L^2({\mathcal{O}}) )^n $, the optimal solution to \eqref{OCP}, with $\ov{\bs{u}},$ $\overline{p} , $ and $\ov{\bs{ \t}}$ as the optimal state, pressure, and control, respectively.
\par
Now, we introduce the limit adjoint system associated with \eqref{6.1}:
Find  a pair $(\ov {\bs{v}}, \ov q) \in (H_0^1({\mathcal{O}}))^n \times L^2({\mathcal{O}})$ which solves the system
\begin{equation}\label{6.4}
				\left\{
				\begin{array}{rlc}
					-\displaystyle \sum\limits_{i, \alpha, \beta =1}^{n} \frac{\partial }{\partial x_{\beta}} \left( b^{ \beta \alpha}_{ji}\ \frac{\partial {\ov{v}_i} }{\partial x_{\alpha}} \right)  +\nabla \ov q &= \ov {\bs u} - \bs {u_d} \quad \text { in }\ {\mathcal{O}}, \\[2mm]
					\operatorname{div}\left(\boldsymbol{\ov{v}}\right) &=0 \quad \text { in } {\mathcal{O}},
				\end{array}
				\right.
\end{equation}
where the tensor  $B^t = (b^{\beta \alpha }_{ji}) = (b^{ \beta \alpha}_{ji})_{1\le i, j, \alpha, \beta \le n}$  is constant, elliptic, and  for $1\le i, j, \alpha, \beta \le n$, is given by
\begin{equation*}\label{6.5}
				b^{ \beta \alpha}_{ji} = a^{ \beta \alpha}_{ji} - \frac{1}{|W^*|} \int_{W^*}	A^t(y) \nabla_{y}\left(\bs P_j^{\beta} -\bs H_j^{\beta}\right)  \colon \nabla_{y} \bs H_i^{\alpha}\, dy,
\end{equation*}\label{6.6}
with $a^{ \beta \alpha}_{ji} = \frac{1}{|W^*|} \int_{W^*}	A^t(y) \nabla_{y}\left(\bs P_j^{\beta} -\bs H_j^{\beta}\right)  \colon \nabla_{y} \bs P_i^{\alpha}\, dy$ as the entries of the constant tensor $A^t_0$. Also, for $1\le j, \beta \le n$, the correctors $ (\bs{H}_j^{\beta}, Z_j^{\beta}) \in (H^1(W^*))^n \times  L^2(W^*)$ solves the cell problem
\begin{equation}\label{6.7}
				\left\{
				\begin{array}{rlc}
					-\displaystyle \operatorname{div}_{y}\left(  A^t(y)\nabla_{y} (\bs P_j^{\beta} - \bs H_j^{\beta} )\right)  + \nabla_{y} Z_j^{\beta}
					&= \bs 0 \quad \text{in } W^*,\\[2mm]
					\displaystyle \bs{\eta} \cdot {A^t(y)}\nabla_{y} (\bs P_j^{\beta} - \bs H_j^{\beta} )  - Z_j^{\beta}  \bs \eta & = \bs 0  \quad \text{on } \p W^*\backslash\p W,\\[2mm]
					\d \operatorname{div}_{y} ( \bs P_j^{\beta} - \bs H_j^{\beta} )  &=0 \quad	\text{in } W^*,\\[2mm]
					\d (\bs H_j^{\beta}, Z_j^{\beta})  & W^* \text{- periodic}, \\[2mm]
					\d \mathcal{M}_{W^*} (\bs H_j^{\beta})  &= \bs0.
				\end{array}
				\right.
\end{equation}
In the following, we state a result similar to Theorem \ref{4t1} that characterizes the optimal control $\overline{\t}$  in terms of the adjoint state $ \overline{\boldsymbol{v}}$ and the proof of which follows analogous to the standard procedure laid in \cite[Chapter 2, Theorem 1.4]{Lions1971}.
\begin{thm}\label{5t2}
Let $\left(\overline{\bs {u}},\overline{p}, \overline{\boldsymbol{\t}}\right) $ be the optimal solution to \eqref{OCP} and $(\bs{\overline{v}}, \ov q)$ be the corresponding adjoint solution to \eqref{6.4}, then the optimal control is characterized by
\begin{equation}\label{c2}
					\overline{\boldsymbol{\t}} = -\frac{1}{\tau} \bs{\overline{v}}  \text {  a.e. in } {\mathcal{O}}.
\end{equation}
Conversely, suppose that a triplet \ $ (\check {\bs u},\check {p}, \check {\bs {\t}}) \in (H_0^1 ({\mathcal{O}}))^n \times L^2({\mathcal{O}}) \times ( L^2({\mathcal{O}}) )^n $ and a pair $\left(\check{\bs v} ,\check{q}\right)  \in (H_0^1 ({\mathcal{O}}))^n \times L^2({\mathcal{O}})$, respectively, satisfy the following systems:
\begin{equation*}
					\left\{
					\begin{array}{rlc}
						-\displaystyle \sum\limits_{j, \alpha, \beta =1}^{n} \frac{\partial }{\partial x_{\alpha}} \left( b^{\alpha \beta}_{ij}\ \frac{\partial {{\check u}_j} }{\partial x_{\beta}} \right)  +\nabla \check p &= -\frac{1}{\tau } \check {\bs v} \quad \text { in }\ {\mathcal{O}}, \\[2mm]
						\operatorname{div}\left(\boldsymbol{\check u}\right) &=0 \quad \text { in } {\mathcal{O}},
					\end{array}
					\right.
\end{equation*}
and
\begin{equation*}
					\left\{
					\begin{array}{rlc}
						-\displaystyle \sum\limits_{i, \alpha, \beta =1}^{n} \frac{\partial }{\partial x_{\beta}} \left( b^{ \beta \alpha}_{ji}\ \frac{\partial {\check{v}_i} }{\partial x_{\alpha}} \right)  +\nabla \check q &= \check {\bs u} - \bs {u_d} \quad \text { in }\ {\mathcal{O}}, \\[2mm]
						\operatorname{div}\left(\boldsymbol{\check{v}}\right) &=0 \quad \text { in } {\mathcal{O}}.
					\end{array}
					\right.
\end{equation*}
Then, the triplet $\left(\check {\bs{u}},\check {p}, -\frac{1}{\tau} \check {\bs v}\right) $\ is the optimal solution to \eqref{OCP}.
\end{thm}
			
\section{Convergence results}\label{Section 7}
We present here the key findings on the convergence analysis of the optimal solutions to the problem  \eqref{OCPE} and its corresponding adjoint system \eqref{4e1} by using the method of periodic unfolding for perforated domains  described in Section \ref{Section 4}.
\begin{thm}\label{6t1}
For given $\ep >0$, let the triplets $(\ov\u_{\bs\ep},\ov p_\ep,\ov\t_{\bs\ep})$ and $(\ov\u,\ov p,\ov\t)$, respectively, be the optimal solutions of the problems  \eqref{OCPE} and \eqref{OCP}. Then
\begin{subequations}\label{yy}
\begin{align}
						&T^*_{\ep} (A_{\ep}) \rightarrow A \quad \text{strongly in}\  (L^2({\mathcal{O}} \times W^*))^{n\times n},\label{yya} \\
						&\widetilde{\ov{\bs{\t}}_{\bs\ep}} \w 	 \Theta\, \ov{\bs{\t}}  \quad \text{weakly in}\ \left( L^2\left({\mathcal{O}} \right) \right)^n,\label{yyf}\\
						&\widetilde{\ov {\bs{u}}_{\bs\ep}} \w \Theta\, \ov{\bs{u}} \quad \text{weakly in}\ (H_0^1({\mathcal{O}}))^n,\label{yyb}\\
						&\widetilde{\ov {\bs{v}}_{\bs\ep}} \w \Theta\, \ov{\bs{v}} \quad \text{weakly in}\ (H_0^1({\mathcal{O}}))^n,\label{yyc}\\
						&\widetilde{\ov {p}_\ep} \w  \frac{\Theta}{n}\, A_0  \nabla {\bs {\ov u}} \colon I + \Theta\, \ov p \quad \text{weakly in}\ L^2({\mathcal{O}}),\label{yyd}\\
						&\widetilde{\ov {q}_\ep} \w  \frac{\Theta}{n}\, A^t_0  \nabla {\bs {\ov v}} \colon I + \Theta\, \ov q \quad \text{weakly in}\ L^2({\mathcal{O}}),\label{yye}
\end{align} 
\end{subequations}
where $A^0$ is a  tensor as defined in Section \ref{Section 6}, $I$ is the ${n \times n}$ identity matrix,  $\displaystyle \ov{\bs{\t}}$ is characterized through \eqref{c2} and the pairs $\left( {\ov{\bs {v}}}_{\bs\ep},	{\ov{{q}}}_\ep \right) $ and $	\left( {\ov{\bs {v}}}, \ov q \right) $  solve respectively the systems \eqref{4e1} and \eqref{6.4}.\\[1mm]
Moreover, 
\begin{equation}\label{J1}
					\lim_{\ep\to0} J_{\ep}(\boldsymbol{\ov \t}_{\bs\ep}) = J(\boldsymbol{\ov \t}).
\end{equation}
\end{thm}
\begin{proof}
First, upon using  Proposition \ref{3p2}\,$\ref{3p26}$ on the entries of the matrix $A_{\ep}$, we obtain \eqref{yya} under the passage of limit $\ep \to 0$. Similarly, one can prove the convergence for the matrix $A^t_{\ep}$ under unfolding.
Next, in view of Theorem \ref{4t4} and the fact that the triplet $(\ov\u_{\bs\ep},\ov p_\ep,\ov\t_{\bs\ep})$ is an optimal solution to problem \eqref{OCPE}, one gets  uniform estimates for the sequences $\{\ov{\bs{\t}}_{\bs\ep}\}$, $\{{\bs{\ov u}}_{\bs\ep}$\}, $\{{\ov p}_\ep\}$, $\{{\bs{\ov v}}_{\bs\ep}\}$, and $\{{\ov q}_\ep\}$ in the spaces $(L^2\left({\mathcal{O}}_{\ep}^*\right))^n$, $(H^1_{\g}({\mathcal{O}}_{\ep}^*))^n$, $L^2\left({\mathcal{O}}_{\ep}^*\right)$, $(H^1_{\g}({\mathcal{O}}_{\ep}^*))^n$, and $L^2\left({\mathcal{O}}_{\ep}^*\right)$, respectively.\\[1mm]
Using the uniform estimate of the sequence $\{\ov{\bs{\t}}_{\bs\ep}\}$   in the space $\left(L^2\left({\mathcal{O}}^*_{\ep} \right)\right) ^n$ and Proposition \ref{3p2}\,$\ref{3p21},$ we have the sequence $\{T^*_{\ep}(\ov{\bs{\t}}_{\bs\ep})\}$ to be  uniformly bounded in the space $\left(L^2\left({\mathcal{O}} \times W^* \right)\right) ^n$. Thus,  by weak compactness, there exists a subsequence not relabelled and a function $\bs{\hat {\t}} $ in $\left(L^2\left({\mathcal{O}} \times W^* \right)\right) ^n$, such that
\begin{equation}\label{7e1}
					T^*_{\ep}(\ov{\bs{\t}}_{\bs\ep}) \w \bs{\hat{\t}} \quad \text{weakly in}\  \left(L^2\left({\mathcal{O}} \times W^* \right)\right) ^n.
\end{equation}
Now, using Proposition \ref{3p2}\,$\ref{3p27}$ in \eqref{7e1} gives
\begin{equation}\label{7.3}
					\widetilde{\ov{\bs{\t}}_{\bs\ep}}  \w \frac{1}{|W|} \int_{W^*}\bs{\hat {\t}}(x,y)\, dy = \Theta \, \bs{\t}_0 \quad \text{weakly in}\  \left(L^2\left({\mathcal{O}} \right)\right) ^n,
\end{equation}
where, $\bs{\t}_0 = \mathcal{M}_{W^*} (\hat{\t})$.\\[1mm]
Employing Proposition \ref{3p2}\,$\ref{3p21}$, we have the uniform boundedness of the sequences $\{T^{\ep} (\ov{\bs{u}}_{\bs\ep})\}$, $\{T^{\ep} (\nabla{\ov{\bs{u}}}_{\bs\ep})\},$ and $\{T^{\ep} (\ov{{p}}_{\ep})\}$ in the respective spaces $ (L^2({\mathcal{O}};  H^1\left( W^* \right)))^n,$ $ (L^2({\mathcal{O}} \times W^*))^{n\times n},$ and $L^2({\mathcal{O}} \times W^*)$. Further, upon  employing Proposition \ref{4p2} and Proposition \ref{3p2}\,$\ref{3p27}$, there exist subsequences not relabelled and  functions $\bs{\hat{u}}$ with $\mathcal{M}_{W^*} (\hat{\u}) = \bs0$, $\bs{u_0},$ and $\hat{p}$ in spaces $ (L^2({\mathcal{O}};  H^1_{per}\left( W^* \right)))^n,$ $ (H^1_0({\mathcal{O}}))^{n},$ and $L^2({\mathcal{O}} \times W^*)$, respectively, such that
\begin{subequations}\label{7.4}
\begin{align}
						&T^*_{\ep} \left( \ov{\bs{u}}_{\bs\ep}\right)  \rightarrow \bs{u_0} \quad \text{strongly in}\  (L^2({\mathcal{O}};  H^1\left( W^* \right)))^n,\label{7.4a}\\
						&T^*_{\ep} \left( \nabla{\ov{\bs{u}}_{\bs\ep}}\right)  \w \nabla{\bs{u_0}} + \nabla_{y}{\bs{\hat u}} \quad \text{weakly in}\  (L^2({\mathcal{O}} \times W^*))^{n \times n},\label{7.4b}\\
						&\widetilde{\ov {\bs{u}}_{\bs\ep}} \w \Theta\, {\bs{u_0}} \quad \text{weakly in}\ (H_0^1({\mathcal{O}}))^n,\label{7.4c}\\
						&T^*_{\ep} \left( \ov{{p}}_{\ep}\right)  \w \hat {p} \quad \text{weakly in}\  L^2({\mathcal{O}} \times W^*),\label{7.4d}\\
						&\widetilde{\ov {p}_\ep} \w  \Theta\, \mathcal{M}_{W^*} (\hat p) \quad \text{weakly in}\ L^2({\mathcal{O}}).\label{7.4e}
\end{align} 
\end{subequations}
Likewise, for the associated adjoint counterparts, viz., $\bs{\ov v}_{\bs\ep}$, and $\ov q_{\ep} $ , one obtains that there exist subsequences not relabelled and  functions $\bs{\hat{v}}$ with  $\mathcal{M}_{W^*} (\hat{\v}) = \bs0$, $\bs{v_0},$ and $\hat{q}$ in spaces $ (L^2({\mathcal{O}};  H^1_{per}\left( W^* \right)))^n,$ $ (H^1_0({\mathcal{O}}))^{n},$ and $L^2({\mathcal{O}} \times W^*)$, respectively, such that
\begin{subequations}\label{7.5}
					\begin{align}
						&T^*_{\ep} \left( \ov{\bs{v}}_{\bs\ep}\right)  \rightarrow \bs{v_0} \quad \text{strongly in}\  (L^2({\mathcal{O}};  H^1\left( W^* \right)))^n,\label{7.5a}\\
						&T^*_{\ep} \left( \nabla{\ov{\bs{v}}_{\bs\ep}}\right)  \w \nabla{\bs{v_0}} + \nabla_{y}{\bs{\hat v}} \quad \text{weakly in}\  (L^2({\mathcal{O}} \times W^*))^{n \times n},\label{7.5b}\\
						&\widetilde{\ov {\bs{v}}_{\bs\ep}} \w \Theta\, {\bs{v_0}} \quad \text{weakly in}\ (H_0^1({\mathcal{O}}))^n,\label{7.5c}\\
						&T^*_{\ep} \left( \ov{{q}}_{\ep}\right)  \w \hat {q} \quad \text{weakly in}\  L^2({\mathcal{O}} \times W^*),\label{7.5d}\\
						&\widetilde{\ov {q}_\ep} \w  \mathcal{M}_{W^*} (\hat q) \quad \text{weakly in}\ L^2({\mathcal{O}}).\label{7.5e}
					\end{align} 
\end{subequations}
The identification of the limit functions $\bs{\hat{u}} $, $\bs{\hat{v}} $, ${\hat{p}} $, ${\hat{q}} $,   $\mathcal{M}_{W^*} (\hat p)$ and $\mathcal{M}_{W^*} (\hat q)$ is carried out in subsequent steps.\\[1mm]
\noindent\textbf{Step 1: (Claim):} For all $\displaystyle \bs \varphi \in (H_{0}^{1}({\mathcal{O}}))^n,\, \bs \psi \in \left(L^{2}\left({\mathcal{O}} ; H_{p e r}^{1}\left(W^{*}\right)\right)\right)^n,$ and $w \in L^2({\mathcal{O}}),$ we claim that the ordered quadruplet $(\bs {u_0},\hat{\bs u},  \hat p, \bs{\t_0})  \in  (H^1_0({\mathcal{O}}))^{n} \times (L^2({\mathcal{O}};  H^1_{per}\left( W^* \right)))^n \times  L^2({\mathcal{O}} \times W^*) \times (L^2({\mathcal{O}}))^n$ is a unique solution to the following limit system:
\begin{equation} \label{7.6}
					\displaystyle \left\{\begin{array}{l}
						\displaystyle \frac{1}{|W|} \int_{{\mathcal{O}} \times W^{*}} A(y)\left(\nabla \bs {u_{0}}+\nabla_{y} \widehat{\bs u}(x,y)\right) \colon \left(\nabla \bs \varphi+\nabla_{y} \bs \psi\right) dx\, dy \\[5mm]- \displaystyle \frac{1}{|W|} \int_{{\mathcal{O}} \times W^{*}} \hat p (x, y) \left( \operatorname{div}(\bs \varphi) + \operatorname{div_{y}}(\bs \psi) \right)  dx\, dy = \Theta \, \int_{{\mathcal{O}}} \bs {\t_0} \cdot \bs \varphi\, dx, \\[6mm]
						and,\displaystyle
						\int_{{\mathcal{O}}} \operatorname{div}(\bs {u_0})\, w\, dx = 0,
					\end{array}\right.
\end{equation}
and the ordered triplet $(\bs {v_0},\hat{\bs v},  \hat q) \in (H^1_0({\mathcal{O}}))^{n}  \times (L^2({\mathcal{O}};  H^1_{per}\left( W^* \right)))^n \times  L^2({\mathcal{O}} \times W^*) $ is a unique solution to the following limit adjoint system:
\begin{equation} \label{7.7}
					\displaystyle \left\{\begin{array}{l}
						\displaystyle \frac{1}{|W|}\int_{{\mathcal{O}} \times W^{*}} A^t(y)\left(\nabla \bs {v_{0}}+\nabla_{y} \widehat{\bs v}(x,y)\right) \colon \left(\nabla \bs \varphi+\nabla_{y} \bs \psi\right) dx\, dy \\[5mm]- \displaystyle \frac{1}{|W|} \int_{{\mathcal{O}} \times W^{*}} \hat q (x, y) \left( \operatorname{div}(\bs \varphi) + \operatorname{div_{y}}(\bs \psi) \right) dx \, dy = \Theta \, \int_{{\mathcal{O}}} (\bs {\u_0} - \bs {u_d}) \cdot \bs \varphi\, dx, \\[6mm]
						and,\displaystyle
						\int_{{\mathcal{O}}} \operatorname{div}(\bs {v_0})\, w\, dx = 0.
					\end{array}\right.
\end{equation}
\noindent\textit{Proof of the Claim:}
Towards the proof of \eqref{7.6}, let us consider a test function ${\bs \varphi} \in (\mathcal{D}({\mathcal{O}}))^n $ in \eqref{2e3} and use properties $\ref{3p21}$, $\ref{3p22}$, and $\ref{3p24}$ of Proposition $\ref{3p2}$ to get
\begin{align}\label{7.8}
					&\frac{1}{|W|} \int_{{\mathcal{O}} \times W^*} T^*_{\ep}(A_{\ep})\, T^*_{\ep}(\nabla{\bs {\ov u}_{\bs\ep}}) \colon T^*_{\ep}(\nabla{\bs \varphi}) \, dx\, dy+ \int_{\hat{\Lambda}_\ep^*} A_{\ep} \nabla{\bs {u}_{\bs\ep}} \colon \nabla {\bs \varphi}\, dx -\int_{\hat{\Lambda}_\ep^*}  p_\ep \operatorname{div}({\bs \varphi}) \, dx \notag \\[1mm]  
					&- \frac{1}{|W|} \int_{{\mathcal{O}} \times W^*} T^*_{\ep}(\ov p_\ep)\ T^*_{\ep}(\operatorname{div}(\bs{\varphi}))\, dx \, dy = \frac{1}{|W|}\int_{{\mathcal{O}\times W^*}}   T^*_{\ep}({{\ov \t_{\bs\ep}}}) \cdot T^*_{\ep}(\bs \phi_{\ep})\, dx\, dy + \int_{\hat{\Lambda}_\ep^*} {\ov \t_{\bs \ep}} \cdot {\bs \varphi}\, dx.
\end{align}
Using Proposition $\ref{3p2}$\,$\ref{3p23}$, the fact that $\lim_{\ep \rightarrow0}|\hat{\Lambda}_\ep^*| = 0,$ and convergences \eqref{7e1}, \eqref{yya}, \eqref{7.4b}, \eqref{7.4d},  we have under the passage of limit $\ep\to0$ in \eqref{7.8}
\begin{align}\label{7.10}
					\displaystyle \frac{1}{|W|} &\int_{{\mathcal{O}} \times W^{*}} A(y)\left(\nabla \bs {u_{0}}+\nabla_{y} \widehat{\bs u}(x,y)\right) \colon \nabla \bs \varphi \, dx\, dy \notag \\[1mm]
					- \displaystyle & \frac{1}{|W|}\int_{{\mathcal{O}} \times W^{*}} \hat p (x, y)  \operatorname{div}(\bs \varphi)\,  dx\, dy = \Theta \int_{{\mathcal{O}}} \bs {\t_0} \cdot \bs \varphi\,  dx, 
\end{align}
which remains valid for every $ \bs \varphi \in (H_0^1({\mathcal{O}}))^n$, by density.\\[1mm]
Now, consider the function $\bs \phi_{\ep}(x) = \ep \phi (x) \bs{\xi}(\frac{x}{\ep})$, where $\phi \in \mathcal{D}({\mathcal{O}})$ and $\bs{\xi} \in (H^1_{per}(W^*))^n$. Employing properties $\ref{3p22}$, $\ref{3p23}$, and $\ref{3p26}$ of Proposition $\ref{3p2}$, one can easily obtain
\begin{subequations}\label{7.12}
					\begin{align}
						&T^*_{\ep} \left( {\bs{\phi}}_{\ep}\right) (x,y) \rightarrow \bs 0 \quad \text{strongly in}\  (L^2({\mathcal{O}} \times W^*))^n,\\
						&T^*_{\ep} \left( \nabla{{\bs{\phi}}_{\ep}}\right)(x,y)  \rightarrow \phi(x) \nabla_{y}{\bs{\xi}(y)} \quad \text{strongly in}\  (L^2({\mathcal{O}} \times W^*))^{n \times n}.
					\end{align}
\end{subequations}
Let us use the  test function $\bs \phi_{\ep}$  in \eqref{2e3} and employ properties $\ref{3p21}$, $\ref{3p22}$, and $\ref{3p24}$ of Proposition  $\ref{3p2}$ to get
\begin{align}\label{ss2}
	 &\frac{1}{|W|}\int_{{\mathcal{O}} \times W^*} T^*_{\ep}(A_{\ep})\, T^*_{\ep}(\nabla{\bs {\ov u}_{\bs\ep}}) \colon T^*_{\ep}(\nabla{\bs \phi_{\ep}}) \, dx\, dy+ \int_{\hat{\Lambda}_\ep^*} A_{\ep} \nabla{\bs {u}_{\bs\ep}} \colon \nabla {\bs \phi_{\ep}}\, dx -\int_{\hat{\Lambda}_\ep^*}  p_\ep \operatorname{div}({\bs \phi_{\ep}}) \, dx \notag \\[1mm]  
	- &\frac{1}{|W|}\int_{{\mathcal{O}} \times W^*} T^*_{\ep}(\ov p_\ep)\ T^*_{\ep}(\operatorname{div}(\bs \phi_{\ep}))\, dx \, dy
	= \frac{1}{|W|}\int_{{\mathcal{O}\times W^*}}   T^*_{\ep}({{\ov \t_{\bs\ep}}}) \cdot T^*_{\ep}(\bs \phi_{\ep})\, dx\, dy + \int_{\hat{\Lambda}_\ep^*} {\ov \t_{\bs\ep}} \cdot \bs \phi_{\ep}\, dx.
\end{align}

In \eqref{ss2}, the absolute value of each integral over $\hat{\Lambda}_\ep^*$ is bounded above with a bound of order $\ep |\hat{\Lambda}_\ep^*|$ or $|\hat{\Lambda}_\ep^*|$. This with the fact that $\lim_{\ep \rightarrow0}|\hat{\Lambda}_\ep^*| = 0$, and convergences \eqref{7e1}, \eqref{yya},  \eqref{7.4b},  \eqref{7.4d}, and  \eqref{7.12}, gives under the passage of limit $\ep\to0$
\begin{equation}\label{7.13}
					\displaystyle \frac{1}{|W|} \int_{{\mathcal{O}} \times W^{*}} A(y)\left(\nabla \bs {u_{0}}+\nabla_{y} \widehat{\bs u}(x,y)\right) \colon \nabla_{y} \bs \psi \, dx\, dy
					- \displaystyle  \frac{1}{|W|}\int_{{\mathcal{O}} \times W^{*}} \hat p (x, y)  \operatorname{div}_{y}(\bs \psi)\,  dx\, dy =  0,
\end{equation}
which remains valid for every $\phi\, \bs \xi  =\bs \psi \in (L^2({\mathcal{O}}; H^1_{per} (W^*)))^n, $ by density.\\[1mm]
Further, for all $w \in L^2({\mathcal{O}})$, we have
\begin{equation}\label{7.15}
					\displaystyle \int_{{\mathcal{O}}^*_{\ep}} \operatorname{div}(\ov{\bs{u}}_{\bs \ep}) w\, dx =  0. 
\end{equation}
Now, upon applying unfolding on \eqref{7.15} and using properties $\ref{3p21}$, $\ref{3p22}$, and $\ref{3p23}$ of Proposition \ref{3p2} along with convergence \eqref{7.4b}, we get under the passage of limit $\ep\to0$
\begin{equation*}\label{7.16}
					\displaystyle \frac{1}{|W|} \int_{{\mathcal{O}} \times W^{*}} \left( \operatorname{div}(\bs{ u_{0}}) + \operatorname{div}_{y}(\bs {\hat u})\right) w\, dx\, dy =  0,
\end{equation*}
which eventually gives upon using the fact that $\bs {\hat u} $ is $W^*-$ periodic, for all $w \in L^2({\mathcal{O}})$:
\begin{equation}\label{7.17}
					\displaystyle \int_{{\mathcal{O}}} \operatorname{div}({\bs{u_0}}) w\, dx =  0.
\end{equation}
Finally, upon adding \eqref{7.10} with \eqref{7.13} and considering \eqref{7.17}, we establish \eqref{7.6}. Likewise, one can easily establish \eqref{7.7}. This settles the proof of the claim.\\[1mm]
				
\noindent\textbf{Step 2:} First, we are going to identify the limit functions  $\bs{\hat{u}} $, $\bs{\hat{v}} $, ${\hat{p}} $, and ${\hat{q}} $. Next, using these identifications, we will identify  $\mathcal{M}_{W^*} (\hat p)$ and $\mathcal{M}_{W^*} (\hat q)$.\\[1mm]
\noindent\textbf{Identification of $\bs{\hat{u}} $, $\bs{\hat{v}} $, ${\hat{p}} $, ${\hat q}$:} Taking sucessively $\bs{\varphi} \equiv 0$ and $\bs{\psi} \equiv 0$ in \eqref{7.6}, yields
\begin{equation}\label{j1}
	\left\{
	\begin{array}{llr}
		&\displaystyle -\operatorname{div}_y (A(y)\nabla_{y} \widehat{\bs u}(x,y)) + \nabla_y {\widehat{ p}(x,y)} = \operatorname{div}_y (A(y))\nabla \bs {u_{0}}(x) \quad \text{in}\ \mathcal{O} \times W^*,\\[2mm]
		& \displaystyle -\operatorname{div}_x \left( \int_{{W^*}} A(y) (\nabla \bs u_{0}(x) + \nabla_{y} \widehat{\bs u}(x,y)) dy \right) + \nabla {\widehat{ p}(x,y)} = |W^*|\, \bs {\t_0} \quad \text{in}\ \mathcal{O}, \\[3mm] 
		& \operatorname{div} (\bs {u_0}) = 0 \quad \text{in}\ \mathcal{O}, \\[2mm] 
		& \widehat{\bs u}(x,\cdot) \quad \text{is}\ W^*-\, \text{periodic}.
	\end{array}
	\right.
\end{equation}
In the first line of \eqref{j1}, we have the $y$-independence of $\nabla {\bs {u_0}}(x)$ and the linearity of operators, viz., divergence and gradient, which suggests $\widehat{\bs u}(x,y)
$
and $ {\widehat{ p}(x,y)}$ to be of the following form (see, for e.g., \cite[Page 15]{S2016}):
\begin{equation}\label{7.18}
	\left\{
	\begin{array}{llr}
		\displaystyle \widehat{\bs u}(x,y) = -\sum\limits_{j, \beta =1}^{n} \bs{\chi}_j^{\beta} (y)\frac{\partial u_{0j} }{\partial x_{\beta}} + \bs{u_1}(x), \\[2mm]
		\displaystyle \widehat{ p}(x,y) = \sum\limits_{j, \beta =1}^{n} \Pi_j^{\beta} (y)\frac{\partial u_{0j} }{\partial x_{\beta}} + {p}_0(x). 
	\end{array}
	\right.
\end{equation}
where the ordered pair $(\bs{u_1}, {p}_0) \in (H^1({\mathcal{O}}))^n   \times L^2({\mathcal{O}})$, and for $1\le j, \beta \le n$, the pair $ (\bs \chi_j^{\beta}, \Pi_j^{\beta})$  satisfy the  cell problem \eqref{6.3}.
Likewise we obtain for the corresponding adjoint weak formulation \eqref{7.7}:
\begin{equation}\label{j2}
	\left\{
	\begin{array}{llr}
		&\displaystyle -\operatorname{div}_y (A(y)\nabla_{y} \widehat{\bs v}(x,y)) + \nabla_y {\widehat{ q}(x,y)} = \operatorname{div}_y (A(y))\nabla \bs {v_{0}}(x) \quad \text{in}\ \mathcal{O} \times W^*,\\[2mm]
		& \displaystyle -\operatorname{div}_x \left( \int_{{W^*}} A(y) (\nabla \bs {v_{0}}(x) + \nabla_{y} \widehat{\bs v}(x,y)) dy \right) + \nabla {\widehat{ q}(x,y)} = |W^*|\, (\bs {\u_0} - \bs u_d) \quad \text{in}\ \mathcal{O}, \\[3mm] 
		& \operatorname{div} (\bs {v_0}) = 0 \quad \text{in}\ \mathcal{O}, \\[2mm] 
		& \widehat{\bs v}(x,\cdot) \quad \text{is}\ W^*-\, \text{periodic},
	\end{array}
	\right.
\end{equation}
and,
\begin{equation}\label{7.19}
					\left\{
					\begin{array}{llr}
						\displaystyle \widehat{\bs v}(x,y) = -\sum\limits_{j, \beta =1}^{n} \bs{H}_j^{\beta} (y)\frac{\partial v_{0j} }{\partial x_{\beta}} + \bs{v_1}(x), \\[2mm]
						\displaystyle \widehat{ q}(x,y) = \sum\limits_{j, \beta =1}^{n} Z_j^{\beta} (y)\frac{\partial v_{0j} }{\partial x_{\beta}} + {q}_0(x), 
					\end{array}
					\right.
\end{equation}
 where the ordered pair $(\bs{v_1}, {q}_0) \in (H^1({\mathcal{O}}))^n   \times L^2({\mathcal{O}})$, and for $1\le j, \beta \le n$, the pair $ (\bs{H}_j^{\beta}, Z_j^{\beta})$  satisfy the  cell problem \eqref{6.7}. \\[1mm]
\noindent\textbf{Identification of $\mathcal{M}_{W^*} (\hat p)$ and $\mathcal{M}_{W^*} (\hat q)$: } Choosing the test function  $ \bs {y} = { (y_1, \dots, y_n)}$  in the weak formulation of \eqref{6.3}, we get
\begin{equation}\label{7.23a}
					\displaystyle \sum\limits_{i, l, k, \alpha =1}^{n} \int_{ W^{*}} a_{lk} \frac{\p}{\p y_k}\left( P_j^{\beta} - {\bs \chi }_j^{\beta}\right) \cdot \frac{\p P_i^{\alpha}}{\p y_l} \frac{\p y_i}{\p y_{\alpha}} \, dy
					= \displaystyle n \int_{ W^{*}} \Pi_j^{\beta} \, dy.
\end{equation}
In view of \eqref{7.4e}, \eqref{7.18}, and \eqref{7.23a}, we observe that
\begin{equation*}
					\displaystyle  \mathcal{M}_{W^*} (\hat p) = \frac{1}{|W^{*}|} \sum\limits_{i, j, l, k, \alpha, \beta =1}^{n} \int_{ W^{*}} a_{lk} \frac{\p}{\p y_k}\left( P_j^{\beta} - {\bs \chi }_j^{\beta}\right) \cdot \frac{\p P_i^{\alpha}}{\p y_l} \frac{\p y_i}{\p y_{\alpha}} \frac{\p { u_{0j}}}{\p x_{\beta}} \, dy + {p}_0, 
\end{equation*}
which upon using the definition of $a_{ij}^{\alpha \beta}$, gives
\begin{equation}\label{eqn7.21}
					\displaystyle  \mathcal{M}_{W^*} (\hat p) = \sum\limits_{i, j, \alpha, \beta =1}^{n}  a_{ij}^{\alpha \beta} \frac{\p { u_{0j}}}{\p x_{\beta}} \frac{\p y_i}{\p y_{\alpha}}   + {p}_0.
\end{equation}
Also, we  re-write the equation \eqref{eqn7.21} to get the identification of $\mathcal{M}_{W^*} (\hat p)$ as
\begin{equation}\label{i1}
					\displaystyle  \mathcal{M}_{W^*} (\hat p) =  A^0 \nabla {\bs {u_0}} \colon I   + {p}_0.
\end{equation}
Likewise, one can obtain the identification  of $\mathcal{M}_{W^*} (\hat q)$ as 
\begin{equation}\label{i2}
					\displaystyle  \mathcal{M}_{W^*} (\hat q) =  A^t_0 \nabla {\bs {v_0}} \colon I   + {q}_0.
\end{equation}
Thus, from \eqref{7.4e} and \eqref{i1}; \eqref{7.5e} and \eqref{i2}, we have the following weak convergences:
\begin{subequations}
\begin{align} 
					\widetilde{\ov {p}_\ep} &\w  \frac{\Theta}{n}\, A_0  \nabla {\bs { u_0}} \colon I + \Theta\,  p_0 \quad \text{weakly in}\ L^2({\mathcal{O}}),\label{PP1}\\
					\widetilde{\ov {q}_\ep} &\w  \frac{\Theta}{n}\, A^t_0  \nabla {\bs {v_0}} \colon I + \Theta\,  q_0 \quad \text{weakly in}\ L^2({\mathcal{O}}).\label{PP2}
\end{align} 
\end{subequations}
			
\noindent\textbf{Step 3: (Claim):} The pairs $(\bs {u_0}, p_0)$ and $(\bs {v_0}, q_0)$ solve the systems \eqref{6.1} and \eqref{6.4}, respectively.\\[1mm]
\noindent\textit{Proof of the Claim:}
We now prove that the pair $(\bs{ u_0}, p_0)$ solves the system \eqref{6.1}. The proof that the pair $(\bs {v_0}, q_0)$  solves the system \eqref{6.4} follows analogously. Substituting the values of $\widehat{\bs u}(x,y)$ and $\widehat{ p}(x,y)$ from expression \eqref{7.18} into equation \eqref{7.10}, we get
\begin{align}\label{7.20}
					\displaystyle \frac{1}{|W|}\sum\limits_{l, k =1}^{n} \int_{{\mathcal{O}} \times W^{*}} a_{lk} \left(\frac{\p {\bs {u_{0}}}}{\p x_k} -  \sum\limits_{j, \beta =1}^{n} \frac{\p {\bs \chi }_j^{\beta}}{\p y_k} \frac{\p { u_{0j}}}{\p x_{\beta}}\right) & \frac{\p {\bs \varphi }}{\p x_l} \, dx\, dy \notag 
					- \displaystyle \frac{1}{|W|}\sum\limits_{j, \beta =1}^{n}\int_{{\mathcal{O}} \times W^{*}} \Pi_j^{\beta} \frac{\p { u_{0j}}}{\p x_{\beta}} \operatorname{div}(\bs \varphi)\,  dx\, dy\\[1mm]
					- \displaystyle & \Theta \int_{{\mathcal{O}}} p_0 \operatorname{div}(\bs \varphi)\,  dx = \Theta \int_{{\mathcal{O}}} \bs {\theta_0} \cdot \bs \varphi\,  dx.
\end{align}
With $P_j^{\beta} = (0, \dots, y_j, \dots, 0)$, we can express the terms $ \frac{\p {\bs u_{0}}}{\p x_k}, \frac{\p {\bs \varphi }}{\p x_l},$ and $\operatorname{div}(\bs \varphi)$ as
\begin{equation*}\label{7.21}
					\begin{array}{llr}
						\displaystyle \frac{\p {\bs {u_0}}}{\p x_k} = \sum\limits_{j, \beta =1}^{n} \frac{\p {P_j^{\beta}}}{\p y_k} \frac{\partial u_{0j} }{\partial x_{\beta}}, \\
						\displaystyle \frac{\p {\bs \varphi }}{\p x_l} = \sum\limits_{i, \alpha =1}^{n} \frac{\p {P_i^{\alpha}}}{\p y_l} \frac{\partial \varphi_{i} }{\partial x_{\alpha}}, \\
						\displaystyle \operatorname{div}(\bs \varphi) = \sum\limits_{i, \alpha =1}^{n} \operatorname{div}_{y}({P_i^{\alpha}}) \frac{\partial \varphi_{i} }{\partial x_{\alpha}}.
					\end{array}
\end{equation*}
Substituting these expressions in \eqref{7.20}, we obtain
\begin{align}\label{7.22}
					\displaystyle & \sum\limits_{i, j, \alpha, \beta =1}^{n}\int_{{\mathcal{O}}} \left(\frac{1}{|W^*|}\sum\limits_{l, k =1}^{n} \int_{ W^{*}} a_{lk} \frac{\p }{\p y_k}\left(  P_j^{\beta} - {\bs \chi }_j^{\beta}    \right) \frac{\p P_i^{\alpha}}{\p y_l}\, dy\right)  \frac{\p { u_{0j}}}{\p x_{\beta}} \frac{\p {\varphi}_i}{\p x_\alpha} \, dx\notag \\[1mm]
					- &\displaystyle \sum\limits_{i, j, \alpha, \beta =1}^{n}\int_{{\mathcal{O}}} \left(\frac{1}{|W^*|} \int_{W^{*}} \Pi_j^{\beta} \operatorname{div}_{y}(P_i^{\alpha}) \, dy \right)\frac{\p { u_{0j}}}{\p x_{\beta}}  \frac{\p {\varphi}_i}{\p x_\alpha} \,  dx
					- \displaystyle  \int_{{\mathcal{O}}} p_0   \operatorname{div}(\bs \varphi)\,  dx = \int_{{\mathcal{O}}} \bs {\t_0} \cdot \bs \varphi\,  dx.
\end{align}		
Now, choosing the test function $ {\bs \chi }_i^{\alpha}$  in the weak formulation of \eqref{6.3}, we get upon using the fact that $ \operatorname{div}_{y}({\bs \chi }_i^{\alpha}) = \operatorname{div}_{y}(P_i^{\alpha})= \delta_{i\alpha}$, where $\delta$ denotes the Kronecker delta function, the following:
\begin{equation}\label{7.23}
					\displaystyle  \int_{ W^{*}} A(y) \nabla_{y}\left( P_j^{\beta} - {\bs \chi }_j^{\beta}\right) \colon \nabla_{y}  {\bs \chi }_i^{\alpha}\, dy
					= \displaystyle  \int_{ W^{*}} \Pi_j^{\beta}  \delta_{i\alpha}\, dy.
\end{equation}
Further, substituting \eqref{7.23} in \eqref{7.22}, we obtain
\begin{align}\label{7.24}
					\displaystyle  \sum\limits_{i, j, \alpha, \beta =1}^{n}\int_{{\mathcal{O}}} \left(\frac{1}{|W^*|}\sum\limits_{l, k =1}^{n} \int_{ W^{*}} a_{lk} \frac{\p }{\p y_k}\left(  P_j^{\beta} - {\bs \chi }_j^{\beta}    \right) \frac{\p }{\p y_l} \left(P_i^{\alpha} - {\bs \chi }_i^{\alpha}\right)\, dy\right)  \frac{\p { u_{0j}}}{\p x_{\beta}} \frac{\p {\varphi}_i}{\p x_\alpha} \, dx\notag \\[1mm]
					- \displaystyle  \int_{{\mathcal{O}}} p_0 \operatorname{div}(\bs \varphi)\,  dx = \int_{{\mathcal{O}}} \bs {\t_0} \cdot \bs \varphi\,  dx.
\end{align}
Also, we can write equation \eqref{7.24} as
\begin{align}\label{7.26}
					\displaystyle  \sum\limits_{i, j, \alpha, \beta =1}^{n}\int_{{\mathcal{O}}} b_{ij}^{\alpha \beta}  \frac{\p { u_{0j}}}{\p x_{\beta}} \frac{\p {\varphi}_i}{\p x_\alpha} \, dx
					- \displaystyle  \int_{{\mathcal{O}}} p_0 \operatorname{div}(\bs \varphi)\,  dx = \int_{{\mathcal{O}}} \bs {\t_0} \cdot \bs \varphi\,  dx,
\end{align}
which holds true for all $\bs{\varphi}\in (H_0^1({\mathcal{O}}))^n$. Also, from equation \eqref{7.17}, we have $\int_{{\mathcal{O}}} \operatorname{div}({\bs{u}}) w\, dx =  0,$ for every $w \in L^2({\mathcal{O}})$. This together with equation  \eqref{7.26} implies that, for $ \bs{\theta} = \bs{\theta_0}$, the pair $(\bs{u_0}, p_0) \in (H_0^1({\mathcal{O}}))^n \times L^2({\mathcal{O}})$ satisfies the variational formulation of the system  \eqref{6.1}.\\[1mm]
\par Therefore, we obtain the optimality system for the minimization problem \eqref{OCP}. Also, in view of  Theorem \ref{5t2}, we conclude  that the triplet $(\bs{u}_0, p_0, \bs{\theta_0})$ is indeed an optimal solution to the problem \eqref{OCP}. Finally, upon considering the optimal solution's uniqueness, we establish that the subsequent pair of triplets are equal:
\begin{equation} \label{7.31}
					(\bs{\ov u}, \ov p, \bs{\ov \t}) = (\bs{u_0}, p_0, \bs{\theta_0}).
\end{equation}
Hence, upon comparing  \eqref{7.4c}, \eqref{7.5c}, \eqref{PP1}, \eqref{PP2}, and \eqref{7.3} with \eqref{7.31}, we obtain convergences  \eqref{yyb}, \eqref{yyc}, \eqref{yyd}, \eqref{yye}, and \eqref{yyf}, respectively. \\[1mm]
\noindent\textbf{Step 4:} Now, we will furnish the proof of the energy convergence for the $L^2-$cost functional.\\[1mm]
Choosing the test function    $({\boldsymbol{\ov u_{\ep}} - \bs{u_d}})$ in the weak formulation of system \eqref{4e1}, we get under unfolding upon passing $\ep\to0$
\begin{align*}
				\lim_{\ep \to 0 }\int_{{\mathcal{O}^*_{\ep}}}  |\boldsymbol {{\ov \u_{\ep}}} -\bs {u_d}|^2\, dx=	\frac{1}{|W|}	\lim_{\ep \to 0 }&\int_{{\mathcal{O}} \times W^*} T^*_{\ep}(A^t_{\ep})\, T^*_{\ep}(\nabla{\bs {\ov v_{\ep}}}) \colon T^*_{\ep}(\nabla({\boldsymbol{\ov u}_{\bs \ep} - \bs{u_d}})) \, dx\, dy \notag\\[1mm]
				&+ \frac{1}{|W|}\lim_{\ep \to 0 }\int_{{\mathcal{O}} \times W^*} T^*_{\ep}(\ov q_\ep)\ T^*_{\ep}(\operatorname{div}({\bs{u_d}}))\, dx \, dy, 
\end{align*}
which gives in view of \eqref{7.31},  Proposition $\ref{3p2}$\,$\ref{3p23}$ and convergences \eqref{7.5a},  \eqref{7.4b}, and \eqref{7.5d}
\begin{align}\label{7.32}
			\lim_{\ep \to 0 }\int_{{\mathcal{O}^*_{\ep}}}  |\boldsymbol {{\ov \u_{\ep}}} -\bs {u_d}|^2\, dx=	\frac{1}{|W|}	& \int_{{\mathcal{O}} \times W^{*}} A^t(y)\left(\nabla \bs {\ov v} +\nabla_{y} \widehat{\bs v}(x,y)\right) \colon \nabla_{y}  \boldsymbol ({{\ov \u}} -\bs {u_d}) \, dx\, dy   \notag\\[1mm]
				&+ \frac{1}{|W|}\int_{{\mathcal{O}} \times W^{*}} \hat q (x, y)  \operatorname{div}(\bs {u_d})\,  dx\, dy.
\end{align}
Also, using  \eqref{7.19} in \eqref{7.32} alongwith \eqref{7.31}, we have upon simplification		
\begin{align}\label{7.33}
			\lim_{\ep \to 0 }\int_{{\mathcal{O}^*_{\ep}}}  |\boldsymbol {{\ov \u_{\ep}}} -\bs {u_d}|^2\, dx = \Theta \left(\sum\limits_{i, j, \alpha, \beta =1}^{n}\int_{{\mathcal{O}}} b_{ji}^{ \beta \alpha}  \frac{\p { \ov{v}_{i}}}{\p x_{\alpha}} \frac{\p {({{\ov u}} - u_d)}_j}{\p x_\beta} \, dx
			- \displaystyle \int_{{\mathcal{O}}} \ov{q} \operatorname{div}(\bs {{\ov \u}} -\bs {u_d})\,  dx\right).
\end{align}
Now, using the test function  $({{\ov \u}} -\bs u_d)$ in the weak formulation of system \eqref{6.4}, we get the following upon comparing  with the right hand side of equation \eqref{7.33}
\begin{equation}\label{7.34}
				\lim_{\ep \to 0 }\int_{{\mathcal{O}^*_{\ep}}}  |\boldsymbol {{\ov \u_{\ep}}} -\bs {u_d}|^2\, dx =\Theta \int_{{\mathcal{O}}}  |\boldsymbol {{\ov \u}} -\bs {u_d}|^2\, dx.
\end{equation} 
Furthermore, in view of \eqref{4e2},  \eqref{7.5a}, and \eqref{7.31}, we get under unfolding upon the passage of limit $\ep \to 0$
		\begin{align}\label{7.36}
		\lim_{\ep \to 0 } \frac{\tau }{2} \int_{{\mathcal{O}^*_{\ep}}}  |\ov{\bs{\t}}_{\bs \ep}|^2\, dx   &=  	\lim_{\ep \to 0 }\frac{1}{2 |W|} \int_{{\mathcal{O} \times W^*}} |T^*_{\ep}(\ov{\bs{\t}}_{\bs \ep})|^2 \, dx \, dy \notag\\[1mm]
		&=\lim_{\ep \to 0 }\frac{1}{2 \tau |W|} \int_{{\mathcal{O} \times W^*}} |T^*_{\ep}(\ov{\bs{v}}_{\bs \ep})|^2 \, dx \, dy\notag\\[1mm]
		 &= \frac{1}{2\tau  |W|} \int_{{\mathcal{O} \times W^*}}   |\ov{\bs v}|^2\, dx \, dy.
	\end{align}
Also, since  $\ov{\bs v}$  is independent of $y$ and comparing the right hand side of \eqref{7.36} with \eqref{c2}, we get 
\begin{equation}\label{7.37}
	\lim_{\ep \to 0 } \frac{\tau }{2} \int_{{\mathcal{O}^*_{\ep}}}  |\ov{\bs{\t}}_{\bs \ep}|^2\, dx =  \frac{\Theta \tau}{2} \int_{{\mathcal{O}}}  | \ov{\t}|^2\, dx.
\end{equation}
Thus, from equations \eqref{7.34} and \eqref{7.37}, we get \eqref{J1}.\\[1mm]	
This completes the proof of Theorem \ref{6t1}.
\end{proof}	
												
\section{Conclusions} 
We have addressed the limiting behavior of an interior OCP corresponding to Stokes equations in an  $n$D $(n\ge 2)$   periodically perforated domain   $\mathcal{O}_ {\ep}^{*}$ via the technique of periodic unfolding in perforated domains (see, \cite{CDGO2006, CDDGZ2012}). We employed the Neumann boundary condition on the part of the boundary of the perforated domain. Firstly, we characterized the optimal control in terms of the adjoint state. Secondly, we deduced the apriori optimal bounds for control, state, pressure, and their associated adjoint state and pressure functions. Thereafter, the limiting analysis for the considered OCP is carried out upon employing the periodic unfolding method in perforated domains. We observed the convergence between the optimal solution to the problem \eqref{OCPE}  posed on the perforated domain $\mathcal{O}_{\ep}^{*}$  and the optimal solution to that of the limit problem \eqref{OCP} governed by stationary Stokes equation posed on a non-perforated domain $\mathcal{O}$. Finally, we established the convergence of energy corresponding to $L^2-$cost functional.

\section{Acknowledgments} \label{sec9}
The first author would like to thank the Ministry of Education, Government of India for Prime Minister's Research Fellowship (PMRF-2900953). The second author would like to thank the support from Science \& Engineering Research Board (SERB) (SRG/2019/000997), Government of India.	
				
\bibliographystyle{amsplain}
													
\bibliography{Ref}
\end{document}